\documentclass[reqno]{amsart}

\usepackage{latexsym}
\usepackage[english]{babel}
\usepackage{fancyhdr}
\usepackage[mathscr]{eucal}
\usepackage{amsmath}
\usepackage{mathrsfs}
\usepackage{amsthm}
\usepackage{amsfonts}
\usepackage{amssymb}
\usepackage{amscd}
\usepackage{graphicx}
\usepackage{graphics}
\usepackage{latexsym}
\usepackage{color}
\usepackage{subfig}
\usepackage{hyperref}
\usepackage{multirow}
\usepackage{tikz}

\DeclareMathOperator*{\argmax}{arg\,max}

\theoremstyle{plain}
\newtheorem{theorem}{Theorem}[section]
\newtheorem{lemma}[theorem]{Lemma}

\newtheorem{proposition}[theorem]{Proposition}

\theoremstyle{definition}
\newtheorem{definition}[theorem]{Definition}

\newtheorem{remark}[theorem]{Remark}
\newtheorem*{remark*}{Remark}
\newtheorem{example}[theorem]{Example}

\numberwithin{equation}{section}

\begin{document}

\title[Card guessing and the birthday problem without replacement]{Card guessing and the birthday problem for sampling without replacement}

\author[J. He]{Jimmy He}
\address[J. He]{Department of Mathematics, Stanford University \\ 450 Jane Stanford Way, Stanford CA 94305 (USA).}
\email{jimmyhe@stanford.edu}
\author[A.~Ottolini]{Andrea Ottolini}
\address[A.~Ottolini]{Department of Mathematics, Stanford University \\ 450 Jane Stanford Way, Stanford CA 94305 (USA).}
\email{ottolini@stanford.edu}
\begin{abstract}
    Consider a uniformly random deck consisting of cards labelled by numbers from $1$ through $n$, possibly with repeats. A guesser guesses the top card, after which it is revealed and removed and the game continues. What is the expected number of correct guesses under the best and worst strategies? We establish sharp asymptotics for both strategies. For the worst case, this answers a recent question of Diaconis, Graham, He and Spiro, who found the correct order \cite{diaconis2020card}. As part of the proof, we study the birthday problem for sampling without replacement using Stein's method. 
\end{abstract}
\date{\today}


\thanks{}
\maketitle
\section{Introduction}
Consider a deck consisting of cards of different types. Throughout this paper, we will tacitly assume that all our decks are thoroughly shuffled. We analyze the following experiment: a person is asked to guess the type of the card on top of the deck, and afterwards they are told the correct type and the card is removed from the deck. This is continued with the second card and so on. The game stops when there are no cards left in the deck. If the player's strategy is to maximize/minimize the expected number of correct guesses, how many cards will be guessed correctly? 

These problems are referred to as \emph{complete feedback games}, since the player has complete knowledge about the past. They have been studied in connection with clinical trials \cite{Blackwell1957, EFRON1971}, and parapsychology experiments \cite{Diaconis1978}. 

In \cite{DG81}, Diaconis and Graham show that for complete feedback games, maximizing (respectively, minimizing) the expected number of correct guesses is achieved with a greedy strategy: the player should guess at each step a type among the ones that occur the most (respectively, the least) in the remaining deck. They provide asymptotics in both scenarios when dealing with decks consisting of $n$ distinct types, each card type occurring $m$ times, when $m$ is large and $n$ is fixed. 

In recent work \cite{diaconis2020card}, similar asymptotic results are obtained in the opposite regime. For $n$ large and $m$ fixed, it is shown that the best strategy gives asymptotically $H_m\ln n$ correct guesses, where $H_m:=1+\ldots +\frac{1}{m}$, and the worst strategy gives order $n^{-\frac{1}{m}}$ correct guesses.

Our work gives much sharper asymptotics for both the best and the worst strategy when $n$ is large and $m$ is fixed. In particular, for the worst strategy, we identify the constant in the leading term, answering a question of Diaconis, Graham, He and Spiro \cite{diaconis2020card}. For the best strategy, we also give asymptotics even when the number of cards of each type is not the same, provided the deck is balanced in a certain sense. Numerical simulations show that for the best strategy, out approximation is very good even for small $n$ as long as $m$ is not too large. For the worst strategy our approximation is good for $m=2$ and $m=3$ but for larger $m$, $n$ needs to be very large.

Given a deck, let $T_j$ be the last time, counting from the bottom of the deck, that no card type appears more than $j$ times. These times are exactly when the best strategy might change what cards are guessed. Our proof proceeds by obtaining good asymptotics for the distribution of the $T_j$'s through studying a version of the birthday problem without replacement. For the worst strategy, the proof is similar. We believe that our results on the birthday problem may be of independent interest. The main tool is a version of Stein's method for Poisson approximation, and we construct the necessary couplings needed to apply the theory.

\subsection{Main results}
We now formally state our main results. A \emph{deck} will be a word in the letters $\{1,\dotsc, n\}$, which we call \emph{types}. Given a deck, let $\bold m=(m_1,\ldots, m_n)$, where $m_i$ denotes the multiplicity of cards of type $i\in\{1,\ldots, n\}$. Denote by $N=\sum m_i$ the total number of cards, by $m=\frac{N}{n}$ the average multiplicity, and by $m^*=\max_{1\leq i\leq n} m_i$ the largest multiplicity. If each type occurs with multiplicity $m$, we write $\bold m=m\bold 1_n$. We call any strategy which maximizes or minimizes the expected number of correct guesses the \emph{best} or \emph{worst} strategy respectively. Let $S_{\bold m}^+$ and $S_{\bold m}^-$ denote the number of correct guesses under the best and worst strategy for a deck with multiplicities $\bold m$. Define
\begin{equation}\label{def: betabalance}
    \gamma_j=\left(\frac{1}{n}\sum_{i=1}^n{m_i\choose j}\right)^{\frac{1}{j}},\quad \beta_j=\frac{1}{n}\sum_{i=1}^n\frac{1}{m^j}{m_i\choose j}=\frac{\gamma_j^j}{m^j}.
\end{equation}
The following theorems gives sharp asymptotics for the expected number of correct guesses under the best and worst strategies. 
\begin{theorem}\label{thm: best case, weakversion}
Consider a deck with $n$ distinct card types, with multiplicities $\bold m=(m_1,\dotsc, m_n)$. Let $\epsilon\in (0,1]$ be the fraction of types that appear with multiplicity $m^*=\max m_i$. Then
\begin{equation*}
    \mathbb{E}[S_{\bold m}^+]=H_{m^*} H_n+\sum_{j=1}^{m^*} \ln \gamma_j+O\left(\ln n\left(\frac{\ln n}{n}\right)^{\frac{1}{m^*}}\right).
\end{equation*}
where $H_n=1+\dotsc\frac{1}{n}$, and the implicit constant depends on $m^*$ and $\epsilon$.
\end{theorem}

When $\mathbf{m}=m\mathbf{1}_n$ (i.e. when all card types have the same multiplicity), Theorem \ref{thm: best case, weakversion} gives the estimate
\begin{equation*}
    \mathbb{E}[S_{\bold m}^+]=H_{m} H_n+\sum_{j=1}^{m} \ln \gamma_j+O\left(\ln n\left(\frac{\ln n}{n}\right)^{\frac{1}{m}}\right),
\end{equation*}
where $\gamma_j={m\choose j}^{\frac{1}{j}}$. We now turn to the worst strategy.

\begin{theorem}
\label{thm: worst case}
Consider a deck with $n$ distinct card types, with multiplicities $\bold m=m\bold 1_n$. Then
\begin{equation*}
    \mathbb E[S^-_{\bold m}]=\sum_{j=\lfloor\frac{m}{2}\rfloor+1}^{m}\frac{\Gamma\left(\frac{j+1}{j}\right)}{\gamma_j n^{\frac{1}{j}}}+O\left(n^{-\frac{2}{m}}\log^2 n\right),
\end{equation*}
where $\Gamma$ denotes the gamma function, and the implicit constant depends on $m$.
\end{theorem}
\begin{remark}
Every deck of the form $\bold m=m\bold 1_n$ satisfies the hypothesis of Theorem \ref{thm: best case, weakversion} with the choice of $\epsilon=1$. In general, the assumption is meant to avoid degenerate cases where one card type dominates the others. 
\end{remark}
\begin{remark}
It is natural to ask what the distribution of $S_{\mathbf{m}}^{\pm}$ is, and in particular whether they satisfy a central limit theorem. While our methods give sharp asymptotics, they do not appear to be able to determine the asymptotic distribution and so we leave this as an open problem.
\end{remark}

\begin{remark}
Theorem \ref{thm: worst case} identifies the leading term as
\begin{equation*}
    \mathbb E(S^-_{\bold m)}\sim \Gamma\left(\frac{m+1}{m}\right)n^{-\frac{1}{m}},
\end{equation*}
answering Question 4.1 of \cite{diaconis2020card}. We remark that while the proof of Theorem \ref{thm: worst case} gives additional terms up to $j=1$, numerical simulations suggest that the error in \ref{thm: worst case} is more or less correct and that there are some additional terms of order about $n^{-\frac{2}{m}}$ not coming from the proof of Theorem \ref{thm: worst case}.
\end{remark}
The proofs of both Theorems \ref{thm: best case, weakversion} and \ref{thm: worst case} follow the same strategy. Let $\bold Z=(Z_1,\ldots, Z_N)$ be the sequence of types extracted, starting from the bottom (in fact, for the worst strategy is is more convenient to start from the top, and the following heuristic has to be adapted). Define, for $1\leq j\leq m^*-1$,
\begin{align*}
    T_j:=\argmax_{1\leq t\leq N}\left\{\max_{1\leq i\leq n}\left|\{1\leq s\leq t: Z_s=i\}\right|\leq j\right\}
\end{align*}
In words, $T_j$ is the last time, starting from the bottom of the deck, that no card type appears more than $j$ times. The $T_j$ are exactly the points at which the card guessed under the best strategy might change, and so $\mathbb{E}[S_{\bold m}^{\pm}]$ can be studied through understanding the distribution of the $T_j$. 

We establish the following result on the distribution of the $T_j$'s. It can be viewed as an analogue of the birthday problem for sampling without replacement, and may be of independent interest.

\begin{theorem}
\label{thm: Tj approx even univariate}
Consider a deck with multiplicities $\bold m$, and fix $j$ with $1\leq j\leq m^*-1$ and $t$ with $1\leq t\leq N$. Assume that for some $\epsilon>0$, we have $m\geq \epsilon m^*$, and the fraction of cards of types that appear with multiplicity at least $\min(2j+1,m^*)$ is at least $\epsilon$. For $\beta$ as in \eqref{def: betabalance}, let
\begin{equation*}
    \lambda=\frac{t^{j+1}}{n^j}\beta_{j+1}.
\end{equation*}
Then
\begin{equation*}
    |\mathbb P(T_j\geq t)-e^{-\lambda}|=O\left(\frac{t}{n}\right)=O\left(\left(\frac{\lambda}{n}\right)^{\frac{1}{j+1}}\right),
\end{equation*}
where the implicit constant depends only on $j$ and $\epsilon$.
\end{theorem}
\begin{remark}
As in Theorem \ref{thm: best case, weakversion}, decks of type $\bold m=m\bold 1_n$ satisfy the hypothesis of Theorem \ref{thm: Tj approx even univariate} with $\epsilon=1$. For general decks, the assumption $\epsilon>0$ prevents one card type from dominating the others.
\end{remark}
For $1\leq j\leq m^*-1$ and $1\leq t\leq N$, define random variables
\begin{equation}\label{def: jtuple}
    W_j(t):=|\left\{1\leq s_1<\ldots<s_{j+1}\leq t: Z_{s_1}=\ldots=Z_{s_{j+1}}\right\}|.
\end{equation}
In words, $W_j(t)$ denotes the number of $(j+1)$-tuples of cards with the same type appearing before time $t$. Directly from the definition,
\begin{align*}
    \{T_j\geq t\}=\{W_j(t)=0\},
\end{align*}
so that Theorem \ref{thm: Tj approx even univariate} is an immediate corollary of the following Poisson approximation result for $W_j(t)$. We let $d_{TV}$ denote the \emph{total variation distance} between two probability measures.
\begin{theorem}
\label{thm: univariatebirthday}
Consider a deck with multiplicities $\bold m$, and fix $j$ with $1\leq j\leq m^*-1$ and $t$ with $1\leq t\leq N$. Assume that for some $\epsilon>0$, we have $m\geq \epsilon m^*$, and the fraction of cards of types that appear with multiplicity at least $\min(2j+1,m^*)$ is at least $\epsilon$. Let
\begin{equation*}
    \lambda=\frac{t^{j+1}}{n^j}\beta_{j+1}.
\end{equation*}
Let $P$ be a Poisson random variable of mean $\lambda$. Then
\begin{align*}
    d_{TV}(W_j(t),P)=O\left(\frac{t}{n}\right),
\end{align*}
where the implicit constant depends only on $j$ and $\epsilon$.
\end{theorem}

Theorem \ref{thm: univariatebirthday} is established using a version of Stein's method for Poisson approximation due to Barbour, Holst and Janson \cite{barbour1992poisson}. As part of the proof, we construct certain couplings of the random deck related to size-bias couplings.

\begin{remark}
While this univariate approximation suffices for our application, the multivariate analogue is no harder and so we also establish this, see Theorem \ref{thm: birthday}. We note that the multivariate analogue has a slightly worse error bound and so is not a strict generalization of the univariate counterpart. In particular, the application to Theorem \ref{thm: worst case} requires the stronger bound.
\end{remark}

\subsection{Card guessing games}
The complete feedback game considered in this paper is an example of a more general collection of partial feedback models for card guessing. These models were considered in \cite{Blackwell1957} and \cite{EFRON1971}, motivated by the study of clinical trials, and in \cite{Diaconis1978}, motivated by tests for extrasensory perception. 

In \cite{DG81}, the mean for the number of correct guesses in the complete feedback model is computed when the number of distinct types is small and the number of cards of each type is large. In particular, in the case of an even deck where $\bold m=m\bold 1_n$, $\mathbb{E}[S^{\pm}_{\bold m}]=m\pm M_n\sqrt{m}+o(\sqrt{m})$, where $M_n$ is some constant depending on $n$. 
The regime studied in this paper, where $n$ is large and $m$ is bounded, was also considered in \cite{DG81} when $m=1$. For general $m$, it was recently studied in \cite{diaconis2020card}, where it was shown that for even decks where $\bold m=m\bold 1_n$,
\begin{equation*}
    \mathbb{E}[S_{\bold m}^+]=(1+o(1))H_m\ln n
\end{equation*}
and
\begin{equation*}
    \mathbb{E}[S_{\bold m}^-]=\Theta\left(n^{-\frac{1}{m}}\right).
\end{equation*}
Our work sharpens both of these results. For the best strategy, numerical simulations show that our approximation performs significantly better, and for small $m$, is good even when $n$ is small.

Uneven decks have also been considered previously in the literature when $m=2$. The mean and limiting distribution for a deck with two card types and possibly different numbers of each was studied in \cite{DG81}, and see also \cite{P91}. Later, in \cite{KP01} and \cite{KPP09}, exact formulas for the distribution of correct guesses was found which gave a complete picture of the limiting distributions as the number of cards of each type varies.

Other versions of the partial feedback model have also been considered. The case of yes-no feedback, where the guesser is only told whether the guess is correct or not, was studied in \cite{DG81}, where it was shown that the greedy strategy is not optimal, and more recently in \cite{diaconis2020card, diaconis2020guessing}. 

One can also study these problems when the deck is not uniformly random, see \cite{C98, P09, L21, KT21}. Recent work of Spiro also studies the case when the deck is shuffled adversarially \cite{S21}.

\subsection{The birthday problem for sampling without replacement}
The classical birthday problem, introduced by von Mises \cite{von1939aufteilungs}, is one of the most well-known and somewhat counter-intuitive facts in elementary probability theory: in a group of $23$ people, it is about fair odds to observe two of them sharing their birthday. 

There have been countless generalizations in various directions, all sharing the following set-up: some discrete process $t\rightarrow \bold H(t)\in \mathbb N^n$ (here $\bold H(t)$ should be thought of as a vector whose $i$th component counts the number of occurrences of $i$) is given and one tries to understand limiting results (for $n$ large) for various features of the top order statistics of the process. These problems have also been studied as urn models.

For instance, the original birthday problem seeks to estimate $\mathbb P(\max \bold H(t)\geq 2)$ in the case where $t\rightarrow \bold H(t)$ is obtained by sampling \emph{without} replacement from a balanced deck with $n$ distinct types (which can be thought as birthdays when $n=365$). In this case, $\bold H(t)$ is distributed according to a symmetric multinomial at each step.

Some generalizations include the study of $\mathbb P(\max \bold H(t)\geq j)$ \cite{Klamkin1967, Holst1985,Holst1986OnBC} (how long before $j+1$ people share the same birthday?), allowing other models for $\bold H(t)$ such as non-symmetric multinomials \cite{Holst1995, Diaconis1989} (unbalanced decks) or mixtures of multinomials \cite{diaconis2002bayesian}(decks with random composition). The case where the process is a function of some underlying graph is also of interest because of its connection with extremal combinatorics and graph colouring \cite{Bhattacharya2017, barbour1992poisson}. Stein's method has been applied to these types problems, see \cite{AGG89} for an application to the original birthday problem.

Our results can be thought of as an analogue where $t\rightarrow \bold H(t)$ consists of sampling without replacement from a deck of cards, i.e., $\bold H(t)$ is distributed as a hypergeometric random variable at all times. Theorem \ref{thm: Tj approx even univariate} states that while sampling without replacement increases the time at which a birthday coincidence occurs -- e.g., for $m=2$ and $n=365$ fair odds of observing two cards with the same type occurs after about $32$ cards have been sampled -- it does not change the scaling with $n$ in the problem (e.g., the first coincidence appears after about $\sqrt n$ samples). The rigidity of the scaling for the birthday problem is something observed also for sampling with replacement from decks with random compositions, see \cite{diaconis2002bayesian} for more details.

\subsection{Outline}
The rest of the paper is structured as follows. In Section \ref{sec: numerics}, we provide some comparisons of our approximations with numerical simulations. In Section \ref{sec: steinsmethod}, we state the version of Stein's method we use and construct the required coupling. In Section \ref{sec: birthday}, we prove results on the birthday problem for sampling without replacement, including the proof of Theorem \ref{thm: univariatebirthday}. Finally, in Sections \ref{sec: best strategy} and \ref{sec: worst strategy}, we apply Theorem \ref{thm: univariatebirthday} to prove Theorems \ref{thm: best case, weakversion} and \ref{thm: worst case} on the card guessing game under the best and worst strategies respectively.

\section{Numerical simulations}
\label{sec: numerics}
In this section, we present some numerical data, comparing the estimates given by Theorems \ref{thm: best case, weakversion} and \ref{thm: worst case} with empirical values. 

\subsection{Best strategy}
Consider an even deck with $\mathbf{m}=m\mathbf{1}_n$ (that is, all card types appear with multiplicity $m$). Numerical simulations show that for small $m$, Theorem \ref{thm: best case, weakversion} gives extremely accurate estimates, even for small values of $n$. 

Figure \ref{fig:max} compares the estimate
\begin{equation}
\label{eq: max estimate}
    \mathbb{E}[S_\mathbf{m}^+]\approx H_mH_n+\sum_{j=1}^m \ln \gamma_j
\end{equation}
with the empirical mean from numerical simulations, computed with 10,000 trials. Data is shown for $m=2,3,4,5,6$ and $n=2,\dotsc, 100$. For $m=2$ and $m=3$, the approximation \eqref{eq: max estimate} is almost indistinguishable from the numerically simulated means for all $n$. Even for $m=6$, the approximation \eqref{eq: max estimate} gives a reasonable estimate, with a relative error of about $2\%$ for $n>30$.

\begin{figure}
    \centering
    \includegraphics[scale=0.7]{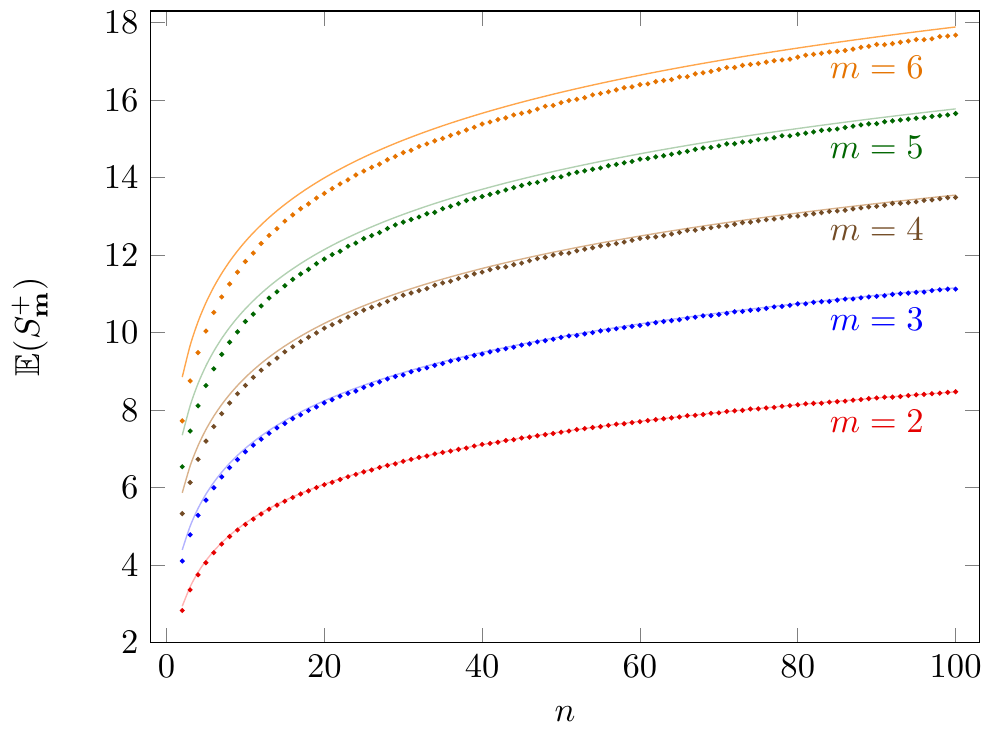}
    \caption{Comparison of approximation of $\mathbb{E}[S_\mathbf{m}^+]$ given by Theorem \ref{thm: best case, weakversion} with simulated mean (10,000 trials), for $\mathbf{m}=m\mathbf{1}_n$ with $m=2,\dotsc,6$, and $n=2,\dotsc, 100$. Dots represent the simulated means and solid lines represent $H_mH_n+\sum \ln\gamma_j$.}
    \label{fig:max}
\end{figure}

Our approximation is much better compared to the one obtained in \cite{diaconis2020card}, which only includes the leading term $H_m\ln n$. Figure \ref{fig:max comparison} shows a comparison of the approximation \eqref{eq: max estimate} with the simulated empirical means as well as the approximation $\mathbb{E}[S_\mathbf{m}^+]\approx H_m\ln n$. Data is shown for $m=4,5,6$ and $n=2,\dotsc, 100$. As expected from the fact that next largest term is of constant order, there is a significant improvement for small $n$. 

\begin{figure}
    \centering
    \includegraphics[scale=0.7]{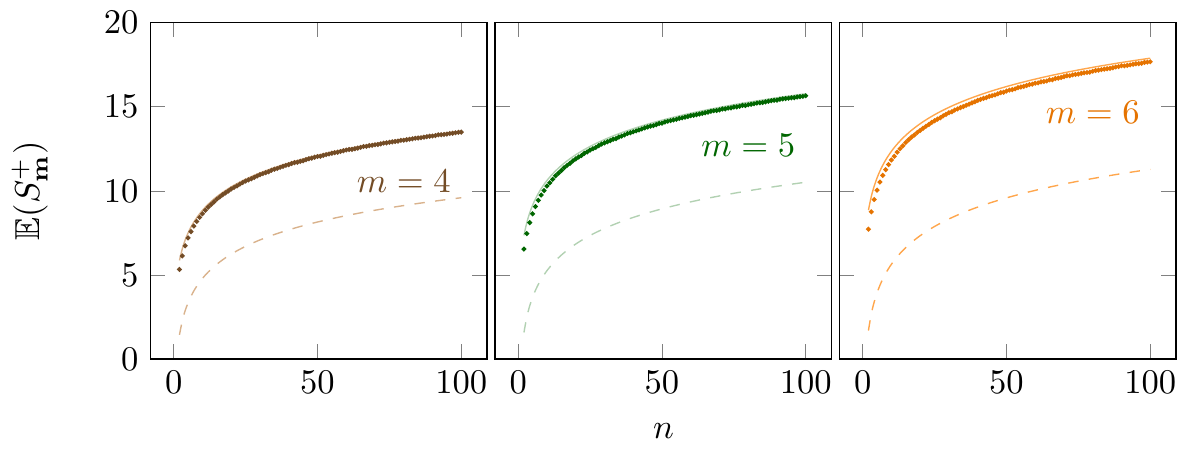}
    \caption{Comparison of approximation of $\mathbb{E}[S_\mathbf{m}^+]$ given by Theorem \ref{thm: best case, weakversion} and the approximation $H_m \ln n$ from \cite{diaconis2020card} with simulated mean (10,000 trials), for $\mathbf{m}=m\mathbf{1}_n$ with $m=4,5,6$, and $n=2,\dotsc, 100$. Dots represent the simulated means, solid lines represent $H_mH_n+m\ln m+\sum \frac{\ln \beta_j}{j}$, and dashed lines represent $H_m \ln n$.}
    \label{fig:max comparison}
\end{figure}

Finally, we also consider uneven decks. Consider a deck with $n$ types, half of which occur with multiplicity $m_1$ and half with multiplicity $m_2$. Figure \ref{fig:max uneven} compares the estimate from Theorem \ref{thm: best case, weakversion} with the empirical mean from numerical simulations, computed with 10,000 trials. Data is shown for $m_1,m_2=6,8,10,12$ and $n=2,4,\dotsc,100$. From the data, it can be seen that the error seems to depend on $\mathbf{m}$ mostly through $m^*$.

\begin{figure}
    \centering
    \includegraphics[scale=0.7]{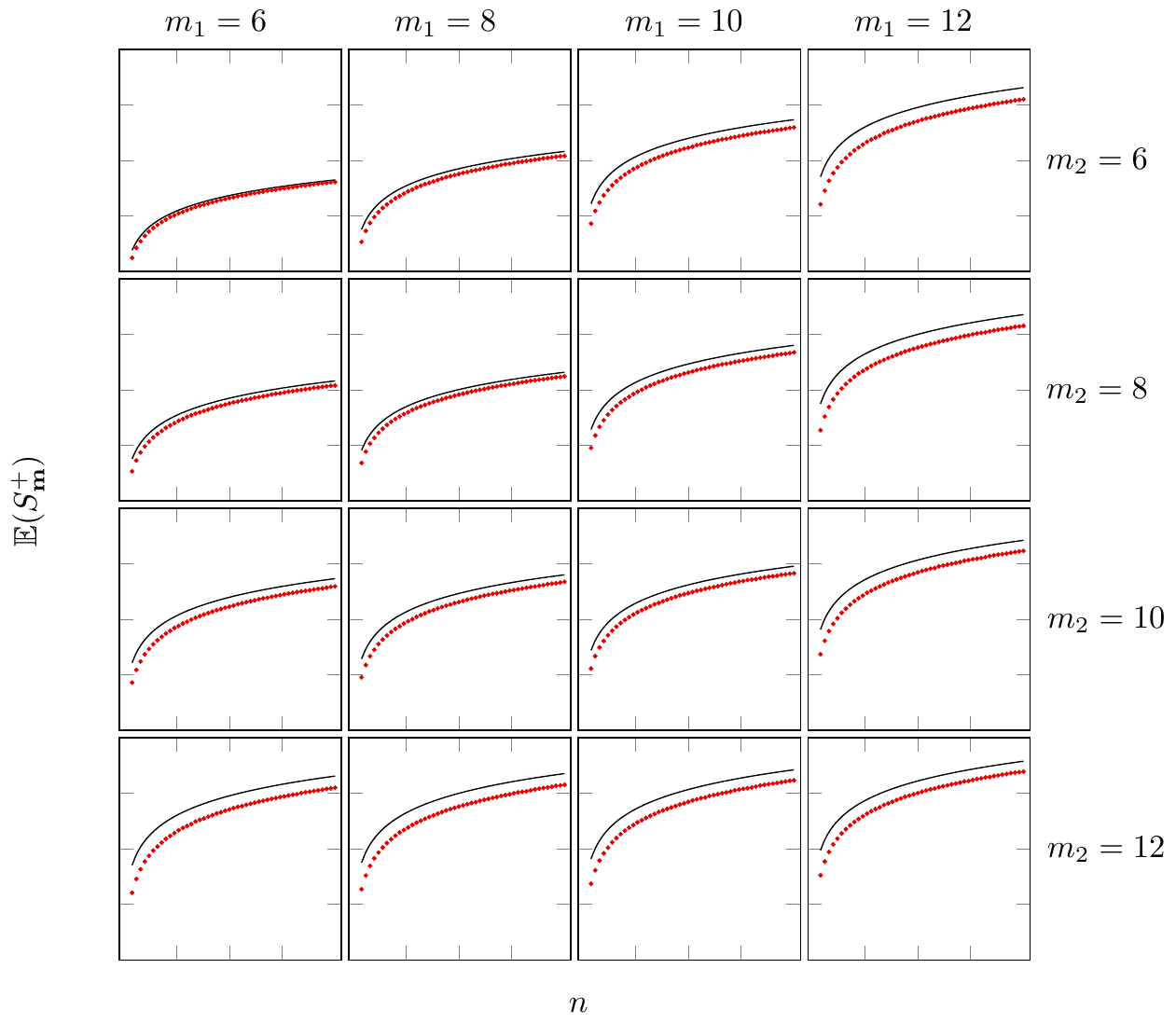}
    \caption{Comparison of approximation of $\mathbb{E}[S_\mathbf{m}^+]$ given by Theorem \ref{thm: best case, weakversion} with simulated mean (10,000 trials), for a deck with $n/2$ types of multiplicity $m_1$ and $n/2$ types of multiplicity $m_2$. Data is shown for $m_1,m_2=6, 8, 10, 12$, and $n=2,4,\dotsc, 100$. Dots represent the simulated means and solid lines represent $H_{m^*}H_n+\sum \ln \gamma_j$.}
    \label{fig:max uneven}
\end{figure}

\subsection{Worst strategy}
Consider an even deck with $\mathbf{m}=m\mathbf{1}_n$. Numerical simulations suggest that for $m=2,3$, Theorem \ref{thm: worst case} gives a good approximation for all reasonably large $n$ (say $n>30$), but that for larger $m$, the approximation deteriorates. This is to be expected from the nature of the error term in Theorem \ref{thm: worst case}.

Figure \ref{fig:min} compares the estimate
\begin{equation}
\label{eq: min estimate}
    \mathbb E[S^-_{\bold m}]\approx\sum_{j=\lfloor\frac{m}{2}\rfloor+1}^{m}\frac{\Gamma\left(\frac{j+1}{j}\right)}{\gamma_jn^{\frac{1}{j}}}
\end{equation}
with the empirical mean from numerical simulations, computed with 10,000 trials. Data is shown for $m=2,3,4,5,6$ and $n=10,20,\dotsc, 1000$. The approximation is good for $m=2$ and $m=3$ when $n$ is reasonably large. However, the approximation gets much worse for larger $m$. For $m=6$, the relative error is about $20\%$ even when $n=1000$. 

Table \ref{tab: min large n} shows the empirical mean from numerical simulations, computed with 10,000 trials, along with the approximation \eqref{eq: min estimate} and the relative error. Data is shown for $m=3,4$ and $n=10,000$, $50,000$, and $100,000$. The data shows that the error does decrease with $n$, albeit quite slowly even for $n=4$. This suggests that if good accuracy is desired, the approximation \eqref{eq: min estimate} is only useful for $m=2,3$.

\begin{figure}
    \centering
    \includegraphics[scale=0.7]{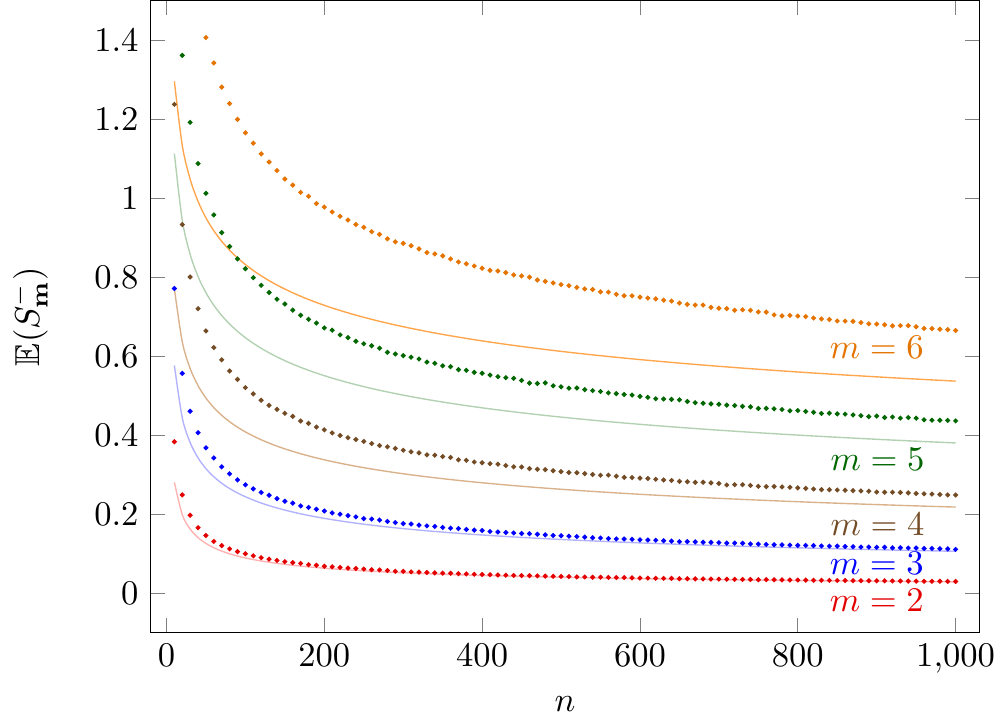}
    \caption{Comparison of approximation of $\mathbb{E}[S_\mathbf{m}^-]$ given by Theorem \ref{thm: worst case} with simulated mean (10,000 trials), for $\mathbf{m}=m\mathbf{1}_n$ with $m=2,\dotsc,6$, and $n=10, 20, \dotsc, 1000$. Dots represent the simulated means and solid lines represent $\sum_{j=\lfloor m/2\rfloor+1}^{m}\Gamma\left(\frac{j+1}{j}\right)\gamma_j^{-1}n^{-\frac{1}{j}}$.}
    \label{fig:min}
\end{figure}

\begin{table}[b]
    \centering
    \begin{tabular}{cl|rrr}
        & & \multicolumn{3}{c}{Number of card types ($=n$)}\\
        & & 10,000&50,000 & 100,000\\\hline
        \multirow{3}{*}{$m=3$}&Approximation \eqref{eq: min estimate}& 0.04657& 0.02653 & 0.02086\\
        & Empirical mean & 0.04729& 0.0269 & 0.02102\\
        & Relative error & 1.53\% & 1.30\%& 0.77\%\\\hline
        \multirow{3}{*}{$m=4$}&Approximation \eqref{eq: min estimate}& 0.11675 & 0.07588 & 0.06309\\
        & Empirical mean & 0.12531 & 0.07963 & 0.06572\\
        & Relative error & 6.83\% & 4.70\% & 4.00\%\\
    \end{tabular}
    \caption{Comparison of approximation of $\mathbb{E}[S_\mathbf{m}^-]$ given by Theorem \ref{thm: worst case} with simulated mean (10,000 trials), for $\mathbf{m}=m\mathbf{1}_n$ with $m=3,4$ and $n=10,000$, $50,000$ and $100,000$.}
    \label{tab: min large n}
\end{table}

The proof of Theorem \ref{thm: worst case} actually gives terms in the sum \eqref{eq: min estimate} all the way to $j=1$, but these extra terms are dominated by the error term. That is, the proof suggests an approximation
\begin{equation}
    \label{eq: min estimate all terms}
    \mathbb E[S^-_{\bold m}]\approx\sum_{j=1}^{m}\frac{\Gamma\left(\frac{j+1}{j}\right)}{\gamma_jn^{\frac{1}{j}}}.
\end{equation}
From numerical simulations, it seems that while the extra terms do help, there is still a significant error that suggests the error term in Theorem \ref{thm: worst case} cannot be significantly improved.

Figure \ref{fig:min extra terms} shows the simulated empirical means compared to both approximations \eqref{eq: min estimate} and \eqref{eq: min estimate all terms}. Data is shown for $m=4,5,6$ and $n=10,20,\dotsc, 1000$. While it seems that these extra terms do help somewhat, it does not seem to affect the order of magnitude of the error. The fact that the extra terms help may be a coincidence due to the fact that the approximation \eqref{eq: min estimate} gives an underestimate and all extra terms are positive. The data suggests that at roughly order $n^{-\frac{2}{m}}$, there is an additional term missing from \eqref{eq: min estimate all terms}.

\begin{figure}
    \centering
    \includegraphics[scale=0.7]{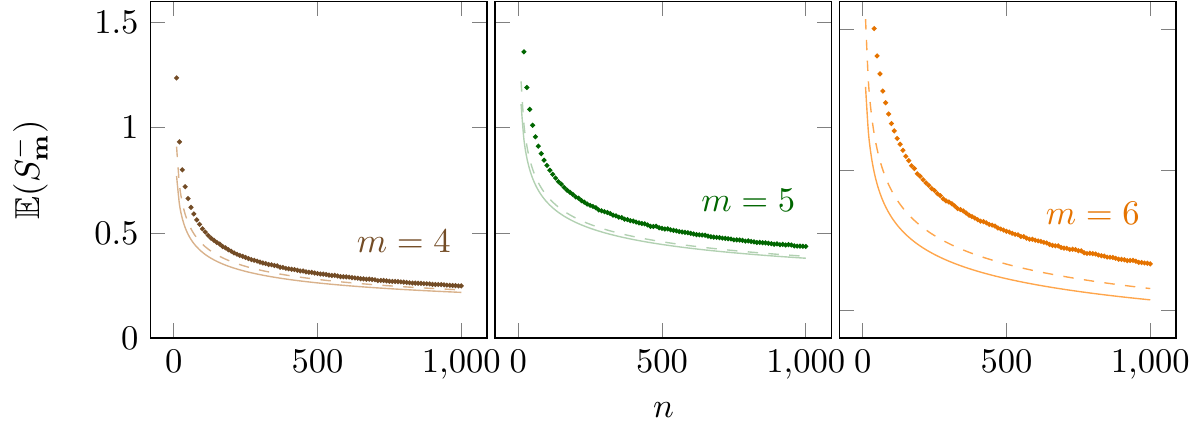}
    \caption{Comparison of approximation of $\mathbb{E}[S_\mathbf{m}^-]$ given by Theorem \ref{thm: worst case} and given by including extra terms from $j=1$ to $\lfloor m/2\rfloor$ with simulated mean (10,000 trials), for $\mathbf{m}=m\mathbf{1}_n$ with $m=2,\dotsc,6$, and $n=10, 20, \dotsc, 1000$. Dots represent the simulated means, solid lines represent $\sum_{j=\lfloor m/2\rfloor+1}^{m}\Gamma\left(\frac{j+1}{j}\right)m^{-1}\gamma_j^{-1}n^{-\frac{1}{j}}$, and dashed lines represent $\sum_{j=1}^{m}\Gamma\left(\frac{j+1}{j}\right)\gamma_j^{-1}n^{-\frac{1}{j}}$.}
    \label{fig:min extra terms}
\end{figure}

\section{Stein's method for Poisson approximation}
\label{sec: steinsmethod}

In this section, we introduce the version of Stein's method for Poisson approximation that is needed as well as its extension to the multivariate setting. The method relies on the construction of certain couplings which are also given in this section.

\subsection{Stein's method}
Initially introduced by Stein \cite{Stein1972ABF} in order to prove normal approximation results, Stein's techniques were developed by Chen \cite{Chen1975} to obtain analogous results for Poisson approximation. For an overview behind the general philosophy of Stein's methods, as well as a large class of examples, we refer the reader to the surveys \cite{Chatterjee2005, Ross2011}. 
We will rely on the following result on multivariate Poisson approximation in order to prove Theorem \ref{thm: birthday}.
\begin{theorem}[{\hspace{1sp}\cite[Theorem 10.J]{barbour1992poisson}}]
\label{thm: multivariate poisson approx}
Let $\Gamma$ be a finite set admitting a partition $\Gamma=\sqcup_{j=1}^k \Gamma_j$. Let $(Y_{\bold s})_{\bold s\in \Gamma}$ be a collection of Bernoulli random variables. Suppose that for each $\bold s, \bold s'\in \Gamma$, $\bold s\neq \bold s'$, we can construct a coupling $Y_{\bold s'}^{\bold s}$ such that the law of the $Y_{\bold s'}^{\bold s}$ is the same as the law of the $Y_{\bold s'}$ conditioned on $Y_{\bold s}=1$. Define $W_1,\ldots, W_k$ by $W_j=\sum _{\bold s\in \Gamma_j}Y_I$. Let $P_1,\dotsc, P_k$ be independent Poisson random variables, with $\mathbb E[P_j]=\mathbb E[W_j]$. Then
\begin{equation*}
    d_{TV}((W_1,\ldots, W_k),(P_1,\ldots, P_k)) \leq \sum_{\bold s\in \Gamma}\left(\mathbb{P}(Y_{\bold s}=1)^2+\mathbb{P}(Y_{\bold s}=1)\sum_{\bold s'\neq \bold s}\mathbb P(Y_{\bold s'}^{\bold s}\neq Y_{\bold s'})\right).
\end{equation*}
\end{theorem}
Our results about card guessing require better control of the marginals when $\mathbb E[W_j]$ is large. We will thus make use of the following consequence of Theorem 2.A of \cite{barbour1992poisson} (see the discussion at the beginning of Section 2.1 of \cite{barbour1992poisson} and apply the triangle inequality)
\begin{theorem}[{\hspace{1sp}\cite[Theorem 2.A]{barbour1992poisson}}]\label{thm: univariate poisson approx}
Let $\Gamma$ be a finite set. Let $(Y_{\bold s})_{\bold s\in \Gamma}$ be a collection of Bernoulli random variables. Suppose that for each $\bold s, \bold s'\in \Gamma$, $\bold s\neq \bold s'$, we can construct a coupling $Y_{\bold s'}^{\bold s}$ such that the law of the $Y_{\bold s'}^{\bold s}$ is the same as the law of the $Y_{\bold s'}$ conditioned on $Y_{\bold s}=1$. Let $W=\sum _{\bold s\in \Gamma}Y_{\bold s}$, and let $P$ be a Poisson random variable, with $\mathbb E[P]=\mathbb E[W]$. Then
\begin{equation*}
    d_{TV}(W, P)\leq \frac{1-e^{-\mathbb E[W]}}{\mathbb E[W]}\sum_{\bold s\in \Gamma}\left(\mathbb{P}(Y_{\bold s}=1)^2+\mathbb{P}(Y_{\bold s}=1)\sum_{\bold s'\neq \bold s}\mathbb P(Y_{\bold s'}^{\bold s}\neq Y_{\bold s})\right).
\end{equation*}
\end{theorem}

Although not explicitly stated, both theorems are related to the notion of size-bias coupling: a random variable $W$ is close to a Poisson if $W^*\approx W+1$, where $W^*$ is the size-bias version of $W$. For more details, see \cite{arratia2019size}.

In order to apply these results, we need to write $W_j(t)$ -- defined in \eqref{def: jtuple} -- as a sum of indicators. Given $j\in\mathbb N$ and $t\in\{1,\ldots, N\}$, define
\begin{align*}
    \Gamma_j(t)=\{\bold s=(s_1,\ldots, s_{j+1}) : 1\leq s_1<\ldots< s_{j+1}\leq t\},
\end{align*}
so that $|\Gamma_j(t)|={t\choose j+1}$ (in particular, it is zero if $t\leq j$). For $\bold s\in \Gamma_j(t)$, define $Y_{\bold s}$ to be the indicator of the event that all cards at times indexed by $\bold s$ are of the same type, i.e., 
\begin{align*}
    Y_{\bold s}=1 \Leftrightarrow Z_{s_1}=\ldots=Z_{s_{j+1}}.
\end{align*}
Directly from the definition,
\begin{align*}
    W_j(t)=\sum_{\bold s\in \Gamma_j(t)}Y_{\bold s}, 
\end{align*}
which is zero if and only if by time $t$ no type has appeared more than $j$ times as we desired. More precisely,
\begin{align*}
    \{T_1\geq t_1,\ldots, T_k\geq t_k\}=\{W_1(t_1)=0,\ldots, W_k(t_k)=0\},
\end{align*}
which allows us to tackle the birthday problem via Poisson approximation. 
\subsection{A general coupling}
Since the $Y_{\bold s}$'s are functions of the deck $\mathbf{Z}$, we instead construct a coupling of the original deck with a new deck $\mathbf{Z}^{\bold s}$ such that the law of $\mathbf{Z}^{\bold s}$ is equal to the law of $\mathbf{Z}$ conditioned on $Y_{\bold s}=1$. 

\begin{definition}
\label{def: coupling}
Let $\bold Z$ be a uniformly random deck of size $N$, and let $\bold s\subseteq \{1,\dotsc, N\}$ be a collection of indices with $|\mathbf{s}|\leq m^*$. We define a random deck $\mathbf{Z}^\mathbf{s}$ coupled with $\mathbf{Z}$ by the following procedure.
\begin{enumerate}
    \item Pick a random card of type $I$ with probabilities
    \begin{align*}
    \mathbb P(I=i)=\mathbb P(Z_{s_1}=i|Z_{s_1}=\ldots=Z_{s_{j+1}}).
    \end{align*}
    This is equivalent to uniformly sampling a card from the cards with multiplicity at least $|\mathbf{s}|$, and then taking $I$ to be the type of that card.
    \item Conditioned on $I=i$, pick a uniformly random $(j+1)$-tuple of times $\bold s^*$ corresponding to cards of type $i$. 
    \item Place the cards at positions $\bold s\setminus \bold s^*$ into positions $\bold s^*\setminus \bold s$, keeping the order the same. Then place the cards that were at positions $\bold s^*\setminus \bold s$ into positions $\bold s\setminus \bold s^*$ (here the order doesn't matter since the cards will all have the same type). Call the resulting deck $Z^{\bold s}$.
\end{enumerate}
We define the random variables $Y_\mathbf{s'}^\mathbf{s}$ as the indicator function that the cards in the deck $\mathbf{Z}^\mathbf{s}$ at times in $\mathbf{s}'$ are all of the same type.
\end{definition}

To show that this coupling satisfies the requirements of Theorem \ref{thm: univariate poisson approx}, we first establish the following lemma.

\begin{lemma}
\label{lem: change one card}
Suppose that $\mathbf{Z}$ is a uniformly random deck with $m_i$ cards of type $i$, for each $i$. If we uniformly pick a card of type $i_1$, and replace it with a card of type $i_2$, then we obtain a uniformly chosen deck containing $m_{i_1}-1$ cards of type $i_1$, $m_{i_2}+1$ cards of type $i_2$, and $m_i$ cards of type $i$ for all other $i$.
\end{lemma}
\begin{proof}
The chance of picking any particular card of type $i_1$ is $1/m_1$. For each possible final deck $\mathbf{Z}'$, there are $m_{i_2}+1$ initial configurations that could lead to $\mathbf{Z}'$. Thus, the probability of seeing $\mathbf{Z'}$ is
\begin{equation*}
    \left(\frac{m_{i_2}+1}{m_{i_1}}\right)\frac{\prod m_i!}{(\sum m_i)!},
\end{equation*}
which is uniform.
\end{proof}

Since the law of $\bold Z$ is exchangeable, the random subset of times $\bold s^*$ is uniformly distributed and independent of the choice of $I$. This can be used to prove the following.

\begin{lemma}\label{lem:sizebias}
Let $bold s\subset \{1,\ldots, N\}$ be a collection of indices with $|\bold s|\leq m^*$. Then, the law of $\bold Z^{\bold s}$ is equal to the law of $\mathbf{Z}$ conditioned on the event $Y_{\bold s}=1$.
\end{lemma}
\begin{proof}
We have
\begin{align*}
    \mathbb P(\bold Z=\cdot|Y_{\bold s}=1)=\sum_{i=1}^n\mathbb P(\bold Z=\cdot|Z_{s_1}=\ldots=Z_{s_{j+1}}=i)\mathbb P(Z_{s_1}=i|Y_{\bold s}=1).
\end{align*}
Therefore, the law of $\bold Z$ conditioned on $Y_{\bold s}=1$ corresponds to first selecting a type $i$ (with the same probability as step one of Definition \ref{def: coupling}), then setting the cards at times $\bold s$ to be be of type $i$, and then sampling without replacement the remaining cards from a deck with $m_{i'}$ cards of types $i'\neq i$ and $m_i-|\mathbf{s}|$ cards of type $i$. The coupled deck $\mathbf{Z}^\mathbf{s}$ is constructed to ensure that the cards at times $\mathbf{s}$ have the same type $i$ (conditioned on $I=i$), so we just have to show that the remainder of the deck away from $\mathbf{s}$ in $\mathbf{Z}^\mathbf{s}$ is uniformly distributed. This follows by further conditioning on the original cards at times $\mathbf{s}$ and repeated applications of Lemma \ref{lem: change one card}.
\end{proof}

Notice that by construction the two decks $\bold Z$ and $\bold Z^{\bold s}$ will coincide at all but at most $2j+2$ times. It is instructive to see an example. 
\begin{example}
Let $\bold m=(4,3,2)$, and
\begin{align*}
    \bold Z=(3,1,2,2,1,3,1,2,1).
\end{align*}
we want to construct $\bold Z^{\bold s}$ for $\bold s=(1,2,3)$. In order to select $I$, notice that there is no $3$-tuple containing $3$, while
\begin{align*}
    \mathbb P(Z_1=1|Z_1=Z_2=Z_3)=\frac{3}{4},\quad \mathbb P(Z_1=2|Z_1=Z_2=Z_3)=\frac{1}{4}.
\end{align*}
Say we select $I=1$, then $\bold s^*$ is uniform among $(2,5,7), (2,5,9), (5,7,9)$. If, for example, we pick $\bold s^*=(2,5,7)$, then we need to switch cards at positions $\bold s\setminus \bold s^*=(1,3)$ with those at positions $\bold s^*\setminus \bold s=(5,7)$. This gives
\begin{align*}
\bold Z^{\bold s}=(1,1,1,2,3,3,2,2,1).
\end{align*}
\end{example}
\subsection{Tail bounds from size-bias couplings}
Since the error in Theorem \ref{thm: Tj approx even univariate} deteriorates as $t$ gets large, we need to control the tail probabilities separately. Tail bounds were obtained in \cite{diaconis2020card}, but these are unfortunately not strong enough to establish Theorem \ref{thm: worst case}. Thus, we use the coupling constructed in Definition \ref{def: coupling} to obtain exponential tail bounds through size-bias coupling.

\begin{definition}
Let $X$ be a discrete non-negative random variable with mean $\mu>0$. A random variable $X^*$ has the \emph{size-bias distribution} with respect to $X$ if $\mathbb{P}(X^*=x)=\mu^{-1}\mathbb{P}(X=x)$ for all $x$. We say that $(X,X^*)$ is a \emph{size-bias coupling} if the pair is defined on a single probability space and $X^*$ has the size-bias distribution with respect to $X$.
\end{definition}

The following result on tail bounds for bounded size-bias coupling will immediately give the required tail bounds. It is a special case of Theorem 3.3 of \cite{CGJ18}.
\begin{theorem}[{\cite[Theorem 3.3]{CGJ18}}]
\label{thm: size-bias tail bound}
Let $X$ be a non-negative random variable with mean $\mu>0$, and suppose that there exists a size-bias coupling $X^*$ such that $\mathbb{P}(X^*\leq X+c)=1$ for some $c>0$. Then for all $0\leq x<\mu$, we have
\begin{equation*}
    \mathbb{P}(X\leq \mu-x)\leq \exp\left(-\frac{x^2}{2c\mu}\right).
\end{equation*}
\end{theorem}

The couplings constructed in Definition \ref{def: coupling} give a size-bias coupling of $W_j(t)$ using the following lemma (and Lemma \ref{lem:sizebias}). The lemma is well-known, see Lemma 3.1 of \cite{CGJ18} for example. We state a special case for Bernoulli random variables, so note that if $X$ is Bernoulli, then its size-bias distribution is $X^*=1$.

\begin{lemma}
\label{lem: size-bias construction}
Let $X=\sum_{i=1}^n X_i$ be a sum of Bernoulli random variables. Suppose that for each $j=1,\dotsc, n$, we have random variables $X_i^{(j)}$ coupled with the $X_i$ so that their distribution is that of the $X_i$ conditioned on $X_j=1$. Define $X^*$ by picking a random index $I$ by $\mathbb{P}(I=i)=\mathbb{E}[X_i]/\mathbb{E}[X]$, and setting $X^*=\sum _{i=1}^n X_i^{I}$. Then $(X,X^*)$ is a size-bias coupling.
\end{lemma}

\begin{lemma}
\label{lem: Tj tail bound}
We have
\begin{equation*}
    \mathbb{P}(T_j\geq t)\leq \exp\left(-\frac{x^2}{2c\mathbb{E}[W_j(t)]}\right)
\end{equation*}
for any $0\leq x<\mathbb{E}[W_j(t)]$, where $c={m^*\choose j+1}$
\end{lemma}
\begin{proof}
Recall that $W_j(t)=\sum_{\mathbf{s}\in\Gamma_j(t)}Y_\mathbf{s}$ denotes the number of $(j+1)$-tuples up to time $t$, $\Gamma_j(t)$ denotes the subsets of $\{1,\dotsc, t\}$ of size $j+1$, and $Y_\mathbf{s}$ denotes the indicator that the subset of indices in $\mathbf{s}$ is a $(j+1)$-tuple. Since the $Y_\mathbf{s}$ are exchangeable, Lemma \ref{lem:sizebias} and Lemma \ref{lem: size-bias construction} imply that if we pick $\mathbf{s'}$ uniformly, and define $W_j(t)^*=\sum_{\mathbf{s}\in\Gamma_j(t)}Y_{\mathbf{s}}^{\mathbf{s}'}$, then $(W_j(t),W_j(t)^*)$ is a size-bias coupling. Moreover, conditioning on the card type $I$ chosen in the construction given by Definition \ref{def: coupling}, we see that any added $(j+1)$-tuple must be of the card type $I$, since only cards of type $I$ can be moved to before time $t$. Since there can be at most $c={m^*\choose j+1}$ many $(j+1)$-tuples of a given type, we have $W_j(t)^*\leq W_j(t)+c$. Theorem \ref{thm: size-bias tail bound} then gives
\begin{equation*}
    \mathbb{P}(W_j(t)=0)\leq \mathbb{P}(W_j(t)\leq \mathbb{E}[W_j(t)]-x)\leq \exp\left(-\frac{x^2}{2c\mathbb{E}[W_j(t)]}\right)
\end{equation*}
for any $0\leq x<\mathbb{E}[W_j(t)]$.
\end{proof}

\section{Birthday problem without replacement}
\label{sec: birthday}
This section is devoted to the proof of Theorem \ref{thm: univariatebirthday} and its multivariate analogue that we state here. Note that the multivariate version gives a slightly weaker bound and so is not a strict generalization. The analogous asymptotic results for sampling with replacement were essentially obtained in \cite{Holst1986OnBC} and \cite{AGG89}.
\begin{theorem}\label{thm: birthday}
Consider a deck with multiplicities $\bold m$ and $1\leq k\leq m^*-1$. Assume that, for some $\epsilon>0$, the fraction of types that appear with multiplicity at least $\min(2k+1,m^*)$ is at least $\epsilon$. Also, assume $m\geq \epsilon m^*$. For any $1\leq t_1\leq \ldots \leq t_k\leq N$, let 
\begin{align*}
    \lambda_j=\frac{t_j^{j+1}\beta_{j+1}}{n^j}, \quad \lambda=\max_{1\leq j\leq k}\lambda_j,
\end{align*}
where $\beta_j$ is defined in \eqref{def: betabalance}.
Then if $P_1,\ldots, P_k$ are independent Poisson random variables with $\mathbb E[P_j]=\lambda_j$, we have
\begin{align*}
    d_{TV}((W_1(t_1),\ldots, W_k(t_k)),(P_1,\ldots, P_k))= O\left((1+\lambda)\left(\frac{\lambda}{n}\right)^{\frac{1}{k+1}}\right)
\end{align*}
where the implicit constant depends on $k$ and $\epsilon$. In particular,
\begin{align*}
    |\mathbb P(T_1\geq t_1,\ldots T_k\geq t_k)-e^{-\sum_{j=1}^k \lambda_j}|=O\left((1+\lambda)\left(\frac{\lambda}{n}\right)^{\frac{1}{k+1}}\right)
\end{align*}
\end{theorem}
\begin{remark}
According to Theorem \ref{thm: univariatebirthday}, as long as one deals with roughly balanced decks, the time at which two identical card types are observed for the first time is of order $n^{\frac{1}{2}}$. On the other hand, the time at which three identical card types are observed is of order $n^{\frac{2}{3}}$. While the two random variables are correlated, it is natural to believe that they are approximately independent, given the gap in their scaling. This is the content of Theorem \ref{thm: birthday}. 
\end{remark}
We now state and prove a sequence of lemmas that will be useful in proving the main results. 
\subsection{Uneven decks}
The usefulness of the assumption on the fraction of cards with maximum multiplicity in our theorems comes from the following lemma.
\begin{lemma}\label{lem: gooddecks}
Let $\bold m$ be a vector of multiplicities consisting of $n$ distinct types and $1\leq j\leq m^*$ (recall $m^*=\max m_i$). Assume that the fraction of types that appear with multiplicity at least $j$ is greater than $\epsilon>0$, and assume that $m\geq \epsilon m^*$.
Then, if $\beta$ is defined as in \eqref{def: betabalance}, one has
\begin{align*}
    \beta_j(\bold m)=\Theta(1),
\end{align*}
where the implicit constant depends on $j$ and $\epsilon$ only.
\end{lemma}
\begin{proof}
Let $I$ be a uniformly chosen index in $\{1,\ldots, n\}$. Then we can write
\begin{equation*}
    \beta_j(\bold m)=m^{-j}\mathbb E\left[{m_I\choose j}\right].
\end{equation*}
We can bound
\begin{align*}
    {m_I\choose j}\leq \frac{m_I^j}{j!}\leq \frac{(m^*)^{j}}{j!},
\end{align*}
and $m\geq \epsilon m^*$ by assumption. Combining the two inequalities,
\begin{align*}
    \beta_j(\bold m)\leq \frac{1}{j!}\frac{1}{\epsilon^j},
\end{align*}
which proves the upper bound. 

As for the lower bound, note that
\begin{equation*}
    \mathbb E\left[{m_I\choose j}\right]\geq \epsilon \mathbb E\left[{m_I\choose j}\,\middle|\, m_I\geq j\right]
\end{equation*}
by assumption, and
\begin{equation*}
    \mathbb{E}\left[{m_I\choose j}\,\middle|\, m_I\geq j\right]\geq \mathbb{E}\left[\left(\frac{m_I}{j}\right)^j\,\middle|\, m_I\geq  j\right]\geq \frac{\mathbb{E}[m_I\mid m_I\geq j]^j}{j^j}
\end{equation*}
by Jensen's inequality. Since $\mathbb{E}[m_I\mid m_I\geq j]\geq m\geq \epsilon m^*$, the result follows.

\end{proof}
\begin{remark}
As already remarked after Theorem \ref{thm: best case, weakversion}, every balanced deck with multiplicities $\bold m=m\bold 1_n$ satisfy the assumption with $\epsilon=1$. 
\end{remark}
\begin{remark}\label{rem: fromttolambda}
Consider the definition of $\lambda$ from Theorem \ref{thm: univariatebirthday}. Lemma \ref{lem: gooddecks} shows
\begin{align*}
    t=\Theta\left(\lambda^{\frac{1}{j+1}}n^{\frac{j}{j+1}}\right),
\end{align*}
a fact we will use repeatedly.
\end{remark}

\subsection{Preliminary bounds}
We start with the following lemma. 
\begin{lemma}\label{lemma:standardbounds}
Let $\bold m$ a deck with $n$ distinct types and $j<m^*$. Assume that for some $\epsilon>0$, at least an $\epsilon$ fraction of them have multiplicity $j+1$. Also, assume that $m\geq \epsilon m^*$. Then we have the following.
\begin{itemize}
    \item For any $\bold s=(s_1,\ldots, s_{j+1})$, we have 
\begin{align*}
\mathbb P(Y_{\bold s}=1)=\Theta\left(\frac{1}{n^j}\right).
\end{align*}
\item For any $1\leq t\leq N$ and corresponding $\lambda=\lambda(t)$ as defined in Theorem \ref{thm: univariatebirthday}, we have 
\begin{align*}
    |\lambda-\mathbb E[W_j(t)]|=O\left(\frac{t^j}{n^j}\right).
\end{align*}
\end{itemize}
All implicit constants depend on $j$ and $\epsilon$ only. 
\end{lemma}
\begin{proof}
Define
\begin{align*}
    g_j(x):=\prod_{i=1}^j\left(1-\frac{i}{x}\right),
\end{align*}
which is an increasing and bounded function for $x>j$, and vanishes on all integer less or equal than $j$. We have
\begin{align*}
    \mathbb P(Y_{\bold s}=1)=\sum_{i=1}^n\frac{m_i\ldots (m_i-j)}{N\ldots (N-j)}=\frac{(j+1)!\beta_{j+1}(\bold m)}{n^jg_{j+1}(N)}.
\end{align*}
Since $N\geq m^*>j$, we obtain
\begin{align*}
    g_{j+1}(N)=\Theta(1),
\end{align*}
and combining with Lemma \ref{lem: gooddecks}, we obtain
\begin{align*}
    \mathbb P(Y_{\bold s}=1)=\Theta\left(\frac{1}{n^j}\right).
\end{align*}
For the second claim, note that
\begin{align*}
    \mathbb E[W_j(t)]=\mathbb P(Y_{\bold s}=1)|\Gamma_j(t)|
\end{align*}
where $\bold s$ is an arbitrary set of times in $\Gamma_j(t)$, and
\begin{align*}
    |\Gamma_j(t)|={t \choose j+1}=\frac{t^{j+1}}{(j+1)!}g_{j+1}(t).
\end{align*}
We consider now two cases
\begin{itemize}
    \item If $t\in \{1,\ldots, j+1\}$, then $\mathbb E[W_j(t)]=0$ and, since $t=\Theta(1)$, we obtain
    \begin{align*}
        |\lambda-\mathbb E[W_j(t)]|= |\lambda|=O\left(\frac{1}{n^j}\right)=O\left(\frac{t^j}{n^j}\right)
    \end{align*}
    since $j>1$. 
    \item For all other $t$'s, we can write
    \begin{align*}
        |\lambda-\mathbb E[W_j(t)]|=\left|\lambda-|\Gamma_j( t)|\mathbb P(Y_\mathbf{s}=1)\right|=\lambda\left|1-\frac{g_{j+1}(t)}{g_{j+1}(N)}\right|
    \end{align*}
    In particular, since $j+1\leq t\leq N$, we can bound
    \begin{align*}
        \left|1-\frac{g_{j+1}(t)}{g_{j+1}(N)}\right|=O\left (1-g_{j+1}(t)\right)=O\left(\frac{1}{t}\right).
    \end{align*}
Using the definition of $\lambda$, this gives
\begin{align*}
    |\lambda-\mathbb E[W_j(t)]|=O\left(\frac{\lambda}{t}\right)=O\left(\frac{t^j}{n^j}\right).
\end{align*}
\end{itemize}
\end{proof}
We now move our attention to the probability of observing a different outcome in two decks coupled according to Lemma \ref{lem:sizebias}. Given two set of times $\bold s$ and $\bold s'$, we indicate by $Y_{\bold s'}^{\bold s}$ the analogue of $Y_{\bold s'}$ obtained from deck $\bold Z^{\bold s}$.

\begin{lemma}\label{lemma:sizebiasindicator}
Let $\bold m$ a deck with $n$ distinct types and $j,l<m^*$. Assume that for some $\epsilon>0$, a fraction $\epsilon$ of the types have multiplicities at least $\min(j+l+1,m^*)$. Also, assume $m\geq \epsilon m^*$.
Then if $|\bold s|=j+1$ and $|\bold s'|=l+1$,
    \begin{align*}
        \mathbb P(Y_{\bold s'}\neq Y_{\bold s'}^{\bold s})=O\left( \frac{1}{n^{l+1-h}}\right),
    \end{align*}
where $h=h_{\bold s,\bold s'}=|\bold s\cap \bold s'|$, and the implicit constant depends on $j,l$ and $\epsilon$.
\end{lemma}
\begin{proof}
If $h>0$, then $\{Y_{\bold s}=1\}\cap \{Y_{\bold s'}=1\}=\{Y_{\bold s\cup \bold s'}=1\}$ where $|\bold s\cup \bold s'|=j+l+2-h\leq j+l+1$. Notice that if $j+l+2-h>m^*$, one has $\mathbb P(Y_{\bold s\cup \bold s'}=1)=0$. Therefore, applying Lemma \ref{lemma:standardbounds} we obtain 
\begin{align*}
    \mathbb P(Y_{\bold s'}\neq Y_{\bold s'}^{\bold s})&\leq \mathbb P(Y_{\bold s'}^{\bold s}=1)+\mathbb P(Y_{\bold s'}=1)\\&= \frac{\mathbb P(Y_{\bold s\cup \bold s'}=1)}{\mathbb P(Y_{\bold s}=1)}+\mathbb P(Y_{\bold s'}=1)
    \\&=O\left( \frac{n^j}{n^{j+l+1-h}}+\frac{1}{n^{l}}\right)\\&=O\left(\frac{1}{n^{l+1-h}}\right)
\end{align*}

If $h=0$, in order for $Y_{\bold s'}\neq Y^{\bold s}_{\bold s'}$, $\bold s^*$ needs to intersect $\bold s'$ (otherwise the two decks coincide at times $\bold s'$). Thus,
\begin{equation*}
    \mathbb P(Y_{\bold s'}\neq Y_{\bold s'}^{\bold s})\leq \mathbb P(Y_{\bold s'}^{\bold s}=1,\mathbf{s}^*\cap\mathbf{s}'\neq \emptyset)+\mathbb P(Y_{\bold s'}=1,\mathbf{s}^*\cap\mathbf{s}'\neq \emptyset),
\end{equation*}
and so we consider the events $\{Y_{\bold s'}=1\}\cap\{\mathbf{s}^*\cap\mathbf{s}'\neq \emptyset\}$ and $\{Y^{\mathbf{s}}_{\bold s'}=1\}\cap\{\mathbf{s}^*\cap\mathbf{s}'\neq \emptyset\}$. Now $\{Y_{\bold s'}=1\}\cap\{\mathbf{s}^*\cap\mathbf{s}'\neq \emptyset\}\cap\{I=i\}$ is contained in the event that $I=i$ and all cards at positions $\mathbf{s}'$ have label $i$. But these two events are independent, and so
\begin{equation*}
\begin{split}
    \mathbb{P}(\{Y_{\bold s'}=1\}\cap\{\mathbf{s}^*\cap\mathbf{s}'\neq \emptyset\}\cap\{I=i\})&\leq \mathbb{P}(I=i)\mathbb{P}(Z_{s}=i ,\forall s\in\mathbf{s}')
    \\&=O\left(\frac{1}{n^{l+1}}\mathbb{P}(I=i)\right),
\end{split}
\end{equation*}
where we used
\begin{align*}
    \mathbb P(Z_s=i,\forall s\in\mathbf{s}')\leq \left(\frac{m^*}{mn}\right)^{l+1}\leq\frac{1}{\epsilon^{l+1}n^{l+1}}=O\left(\frac{1}{n^{l+1}}\right).
\end{align*}
Summing over $i$ gives $\mathbb{P}(\{Y_{\bold s'}=1\}\cap\{\mathbf{s}^*\cap\mathbf{s}'\neq \emptyset\})=O\left(\frac{1}{n^{l+1}}\right)$.

As for $\{Y^{\mathbf{s}}_{\bold s'}=1\}\cap\{\mathbf{s}^*\cap\mathbf{s}'\neq \emptyset\}\cap\{I=i\}$, it is contained in the event $\{I=i\}\cap E$, where $E$ is the event that at least one card in positions $\mathbf{s}'$ is of type $i$, and also that there are at least $l+1$ cards of the same type among those in positions $\mathbf{s}\cup\mathbf{s}'$. Now by a union bound, it suffices to control the probability that $l+1$ specific positions among $\mathbf{s}\cup\mathbf{s}'$ have the same card type, and a specific position has card type $i$, and this gives $\mathbb{P}(E)=O\left(\frac{1}{n^{l+1}}\right)$ (say by further conditioning on the card type of the $(l+1)$-tuple). Since $E$ is independent of $I$, summing over $i$ gives $\mathbb{P}(\{Y^{\mathbf{s}}_{\bold s'}=1\}\cap\{\mathbf{s}^*\cap\mathbf{s}'\neq \emptyset\})=O\left(\frac{1}{n^{l+1}}\right)$.

\end{proof}
The last ingredient is a bound on the number of overlapping times. To this aim, for $j,l<m^*$ and two integers $1\leq t_j, t_l\leq N$, define $\Gamma_{j,l,h}(t_j,t_l)$ to be the set of times $\bold s\in \Gamma_j(t_j)$ and $\bold s'\in \Gamma_l(t_l)$ such that $\bold s\neq \bold s'$ and $|\bold s\cap \bold s'|=h$. Then we have the following easy lemma.
\begin{lemma}\label{lem: overlappingtimes}
In the notation above, if $u=\min(j,l)$ and $v=\max(j,l)$, 
\begin{align*}
    \Gamma_{j,l,h}(t_j,t_l)=O\left(t_u^{u+1}t_v^{v+1-h}\right)
\end{align*}
with the implicit constant depending on $j,l$.
\end{lemma}
\begin{proof}
A counting argument shows
\begin{align*}
    \Gamma_{j,l,h}(t_j,t_l)\leq{t_u\choose h}{t_u-(u+1) \choose u+1-h}{t_v-(v+1) \choose v+1-h}.
\end{align*}
Using the bound ${t\choose j}=O\left(t^j\right)$, the claim follows. 
\end{proof}

\subsection{Proof of Theorems \ref{thm: univariatebirthday} and \ref{thm: birthday}}

\begin{proof}[Proof of Theorem \ref{thm: univariatebirthday}]
First, notice that for $t\geq n$ there is nothing to prove, so we assume $t\leq n$. 
We have
\begin{equation}\label{triangular}
    d_{TV}(W_j(t),P)\leq d_{TV}(W_j(t),\widetilde P)+d_{TV}(\widetilde P, P),
\end{equation}
where $\widetilde P$ is a Poisson random variable with mean $\mathbb E\left[W_j(t)\right]$ (note here the degenerate situation when $\mathbb{E}[W_j(t)]=0$ poses no issues). Since the total variation distance between two Poisson random variables is bounded by the difference of their means, Lemma \ref{lemma:standardbounds} shows
\begin{align*}
    d_{TV}(P,\widetilde P)\leq |\lambda-\mathbb E[W]|=O\left(\frac{t^j}{n^j}\right)=O\left(\frac{t}{n}\right)
\end{align*}
for all $t\leq n$, since $j\geq 1$. This takes care of the second summand in \eqref{triangular}. 

As for the first summand, note that if $t\leq j+1$, then $W_j(t)=0$, so we can assume $t>j+1$. Then the sum can be bounded by means of Theorem \ref{thm: univariate poisson approx}, which gives
\begin{equation*}
    d_{TV}(W_j(t),\widetilde P)\leq \frac{1-e^{\mathbb{E}[W_j(t)]}}{\mathbb{E}[W_j(t)]}\left(\sum_{\bold s\in\Gamma_j(t)}\mathbb P(Y_{\bold s}=1)^2+\sum_{\substack{\bold s, \bold s'\in\Gamma_j(t)\\\mathbf{s}\neq \mathbf{s}'}}\mathbb P(Y_{\bold s}=1)\mathbb P(Y_{\bold s'}^{\bold s}\neq Y_{\bold s'})\right).
\end{equation*}
For the first sum, using Lemma \ref{lemma:standardbounds} and $\Gamma_j(t)=O(t^j)$, we have
\begin{align*}
    \sum_{\bold s\in\Gamma_j(t)}\mathbb P(Y_{\bold s}=1)^2=|\Gamma_j(t)|\mathbb P(Y_{\bold s}=1)^2=O\left(\frac{t^{j+1}}{n^{2j}}\right)=O\left(\frac{\lambda}{n^j}\right).
\end{align*}
For the second sum, we note that $\mathbb P(Y_{\bold s'}^{\bold s}\neq Y_{\bold s'})$ depends only on $|\mathbf{s}\cap\mathbf{s}'|$, and so 
\begin{equation*}
    \sum_{\substack{\bold s, \bold s'\in\Gamma_j(t)\\\mathbf{s}\neq \mathbf{s}'}}\mathbb P(Y_{\bold s}=1)\mathbb P(Y_{\bold s'}^{\bold s}\neq Y_{\bold s'})=\sum_{h=0}^j|\Gamma_{j,j,h}(t,t)|\mathbb P(Y_{\bold s}=1)\mathbb P(Y_{\bold s'}^{\bold s}\neq Y_{\bold s'}),
\end{equation*}
where on the right hand side, $\mathbf{s}$ and $\mathbf{s}'$ satisfy $|\mathbf{s}\cap\mathbf{s}'|=h$ but are otherwise arbitrary. Using Lemma \ref{lemma:sizebiasindicator} and Lemma \ref{lem: overlappingtimes} with $l=j$, 
\begin{equation*}
\begin{split}
    \sum_{h=0}^j|\Gamma_{j,j,h}(t,t)|\mathbb P(Y_{\bold s}=1)\mathbb P(Y_{\bold s'}^{\bold s}\neq Y_{\bold s'})&=\sum_{h=0}^jO\left(\frac{t^{2j+2-h}}{n^{2j+1-h}}\right)\\&=O\left(\frac{t^{j+2}}{n^{j+1}}\right)\\&=O\left(\frac{\lambda t}{n}\right),
\end{split}
\end{equation*}
where in the second last step, we used that $t\leq n$, while the last step comes from the definition of $\lambda$ and Lemma \ref{lem: gooddecks}. 

Finally, the function $f(x)=\frac{1-e^{-x}}{x}$ satisfies
\begin{equation*}
    |f(x)-f(y)|\leq \frac{|x-y|}{\max(x,1)^2}
\end{equation*}
for $0<x<y$. Lemma \ref{lemma:standardbounds} gives $|\lambda-\mathbb{E}[W_j(t)]|=O\left(\frac{t^j}{n^j}\right)$. If $t^{2j+1}\leq n^{2j}$, then $\frac{t^j}{n^j}=O(\lambda^{-1})$ and so
\begin{align*}
    \frac{1-e^{-\mathbb E\left[W_j(t)\right]}}{\mathbb E[W_j(t)]}=\frac{1-e^{-\lambda}}{\lambda}+O\left(\frac{t^j}{n^j}\right)=O\left(\frac{1}{\lambda}\right).
\end{align*}
Otherwise, $t^{2j+1}\geq n^{2j}$, and since we assumed $t\leq n$, $|\lambda-\mathbb{E}[W_j(t)]|=O(1)$, and so $|f(\lambda)-f(\mathbb{E}[W_j(t)])|=O(\lambda^{-2})$. This gives
\begin{align*}
    \frac{1-e^{-\mathbb E\left[W_j(t)\right]}}{\mathbb E[W_j(t)]}=O\left(\frac{1}{\lambda}\right).
\end{align*}

Therefore, combining all the bounds gives
\begin{align*}
       d_{TV}(W_j(t),\widetilde P)=O\left(\frac{t}{n}\right).
\end{align*}
The second part of the statement follows at once from $t=O\left(\lambda^{\frac{1}{j+1}}n^{\frac{j}{j+1}}\right)$.
\end{proof}

The proof of Theorem \ref{thm: birthday} is fairly similar. We include it here for completeness.

\begin{proof}[Proof of Theorem \ref{thm: birthday}]
It suffices to show the theorem for $\lambda\leq n$. As in the proof of Theorem \ref{thm: univariatebirthday}, the triangle inequality reduces the problem to bounding
\begin{equation*}
    d_{TV}((W_1(t_1),\ldots, W_k(t_k)),(P_1,\ldots, P_k))  
\end{equation*}
and
\begin{equation*}
    d_{TV}((\widetilde P_1,\ldots, \widetilde P_k),(P_1,\ldots, P_k)),
\end{equation*}
where the $\widetilde P_j$'s are Poisson with means $\mathbb E\left[W_j(t)\right]$. Since total variation distance between product measures is controlled by total variation distance of the marginals, the same argument used in the proof of Theorem \ref{thm: univariatebirthday} gives
\begin{align*}
    d_{TV}((P_1,\ldots, P_k)),(\widetilde P_1,\ldots, \widetilde P_k))&\leq \sum_{j=1}^kd_{TV}(P_j,\widetilde P_j)\\&=\sum_{j=1}^kO\left(\frac{t_j^j}{n^j}\right)\\&=\sum_{j=1}^kO\left(\left(\frac{\lambda}{n}\right)^{\frac{j}{j+1}}\right)\\&=O\left(\left(\frac{\lambda}{n}\right)^{\frac{1}{2}}\right).
\end{align*}
Again, we may assume that $t_j>j+1$ for all $j$ as otherwise $W_j(t_j)=0$.

In order to bound the other term, we use Theorem \ref{thm: multivariate poisson approx}. If $\Gamma=\cup_{1\leq j\leq k}\Gamma_{j}(t_j)$, then Lemma \ref{lemma:standardbounds} and the fact that $|\Gamma_{j}(t_j)|=O\left(t_j^{j+1}\right)$ gives
\begin{align*}
    \sum_{\bold s\in \Gamma}\mathbb P(Y_{\bold s}=1)^2=\sum_{j=1}^{k}|\Gamma_{j}(t_j)|\mathbb P(Y_{\bold s}=1)^2=\sum_{j=1}^{k}O\left(\frac{t^{j+1}_j}{n^{2j}}\right)=O\left(\frac{\lambda}{n}\right),
\end{align*}
where the summands depend only on $j=|\mathbf{s}|-1$.
For the second part of the bound, we note that $\mathbb P(Y_{\bold s}=1)\mathbb P(Y_{\bold s'}^{\bold s}\neq Y_{\bold s'})$ depends only on $j=|\mathbf{s}|-1$, $l=|\mathbf{s}'|-1$ and $h=|\mathbf{s}\cap\mathbf{s}'|$. Then we write
\begin{align*}
\sum_{\substack{\bold s, \bold s'\in\Gamma\\\mathbf{s}\neq\mathbf{s}'}}\mathbb P(Y_{\bold s}=1)\mathbb P(Y_{\bold s'}^{\bold s}\neq Y_{\bold s'})&=\sum_{j,l=1}^{k}\sum_{h=0}^{\min(j,l)}|\Gamma_{j,l,h}(t_j,t_l)|\mathbb P(Y_{\bold s}=1)\mathbb P(Y_{\bold s'}^{\bold s}\neq Y_{\bold s'}).
\end{align*}
By means of Lemma \ref{lemma:standardbounds}, Lemma \ref{lemma:sizebiasindicator} and Lemma \ref{lem: overlappingtimes}, together with the assumption $t_u\leq t_v$ for $u\leq v$, we obtain (here $u=\min(j,l), v=\max(j,l)$)
\begin{align*}
    |\Gamma_{j,l,h}(t_j,t_l)|\mathbb P(Y_{\bold s}=1)\mathbb P(Y_{\bold s'}^{\bold s}\neq Y_{\bold s'})&=O\left( \frac{t_v^{u+v+2-h}}{n^{u+v+1-h}}\right)\\&=O\left(\frac{\lambda^{\frac{u+v+2-h}{v+1}}}{n^{\frac{u-h+1}{v+1}}}\right)\\&=O\left(\lambda\left(\frac{\lambda}{n}\right)^{\frac{u-h+1}{v+1}}\right),
\end{align*}
where we used $t_v=O\left(\lambda^{\frac{1}{v+1}}n^{\frac{v}{v+1}}\right)$. For given $v$, the assumption $\lambda\leq n$ guarantees that the maximum is achieved for $u=h$. Therefore, we obtain
\begin{align*}
\sum_{\substack{\bold s, \bold s'\in\Gamma\\\mathbf{s}\neq\mathbf{s}'}}\mathbb P(Y_{\bold s}=1)\mathbb P(Y_{\bold s'}^{\bold s}\neq Y_{\bold s'})=O\left(\lambda\left(\frac{\lambda}{n}\right)^{\frac{1}{k+1}}\right).
\end{align*}
Combining all bounds and using Theorem \ref{thm: multivariate poisson approx}, 
\begin{align*}
    d_{TV}((W_1,\ldots, W_k),(P_1,\ldots, P_k))&=O\left(\frac{\lambda}{n}+\left(\frac{\lambda}{n}\right)^{\frac{1}{2}}+\lambda\left(\frac{\lambda}{n}\right)^{\frac{1}{k+1}}\right)\\&=O\left((1+\lambda)\left(\frac{\lambda}{n}\right)^{\frac{1}{k+1}}\right),
\end{align*}
where the last step follows from the assumption $\lambda\leq n$. 
\end{proof}
\section{Best strategy: Proof of Theorem \ref{thm: best case, weakversion}}
\label{sec: best strategy}
In this section, we prove Theorem \ref{thm: best case, weakversion}, and in the next section, we prove Theorem \ref{thm: worst case}. The proofs of both Theorems \ref{thm: best case, weakversion} and \ref{thm: worst case} are similar in spirit, but due to some technical details, the proofs proceed along different lines. Before we proceed with the proofs, we first give a sketch and outline the key differences.

The main idea is that both $\mathbb{E}[S_\mathbf{m}^+]$ and $\mathbb{E}[S_\mathbf{m}^-]$ have simple formulas in terms of the $T_j$, with (here, assume that $\mathbf{m}=m\mathbf{1}_n$)
\begin{equation}
\label{eq: best case formula}
    \mathbb{E}[S_\mathbf{m}^+]=\sum_{j=0}^{m-1}\sum_{t=1}^N\frac{1-\mathbb{P}(T_j\geq t)}{t}
\end{equation}
and
\begin{equation}
\label{eq: worst case formula}
    \mathbb{E}[S_\mathbf{m}^-]=\sum_{j=0}^{m-1}\sum_{t=1}^N\frac{\mathbb{P}(T_j\geq t)}{nm-t+1}.
\end{equation}
We then wish to use Theorem \ref{thm: Tj approx even univariate} to approximate these sums and replace the resulting sums with integrals, which will give the desired asymptotics.

Since the error in Theorem \ref{thm: Tj approx even univariate} deteriorates for large $t$, we must cut off the sum. This is the main difficulty in the proof of Theorem \ref{thm: worst case}, since in the tail, the denominators in \eqref{eq: worst case formula} becomes small.

For Theorem \ref{thm: best case, weakversion}, the main difficulty is actually in studying the resulting integral, which has a singularity at $0$, reflecting the fact that the denominators in \eqref{eq: best case formula} are small when $t$ is small. While it would be possible to directly study this integral, we prefer to give a more probabilistic proof, which mostly avoids these issues.

\begin{proof}[Proof of Theorem \ref{thm: best case, weakversion}]
We proceed as in \cite{diaconis2020guessing}. Write $ S^+_{\bold m}$ as
\begin{align*}
    S^+_{\bold m}=\sum_{t=1}^N X_t,
\end{align*}
where $X_t$ is the indicator that the $t$-th card from the bottom is guessed correctly. Let $J_t$ denote the largest multiplicity of a card among the first $t$ cards (counting from the bottom). Notice $\{J_t>j\}=\{T_{j}<t\}$. Under the best strategy, we have
\begin{align*}
    \mathbb P(X_t=1|J_t=j)=\frac{j}{t},
\end{align*}
so that we can write
\begin{align*}
    \mathbb P(X_t=1)=\sum_{j=1}^{m^*}\frac{j}{t}\mathbb P(J_t=j)=\sum_{j=0}^{m^*-1}\frac{\mathbb P(J_t> j)}{t}=\sum_{j=0}^{m^*-1}\frac{\mathbb P(T_j<t)}{t},
\end{align*}
where we take $T_0=0$. 
Summing over $t$ and exploiting linearity, 
\begin{align*}
    \mathbb E[S_{\bold m}^+]&=\sum_{t=1}^{N}\sum_{j=0}^{m^*-1}\frac{\mathbb P(T_j<t)}{t}\\&=\sum_{j=0}^{m^*-1}\sum_{t=1}^N\sum_{s=1}^{t-1}\frac{\mathbb P(T_j=s)}{t}\\&=\sum_{j=0}^{m^*-1}\sum_{s=1}^N\sum_{t=s+1}^N\frac{\mathbb P(T_j=s)}{t}\\&=\sum_{j=0}^{m^*-1}\sum_{s=1}^N(H_N-H_s)\mathbb P(T_j=s)\\&=m^*H_N-\sum_{j=1}^{m^*-1}\mathbb E\left[H_{T_j}\right],
\end{align*}
where $H_i=1+\ldots+\frac{1}{i}$ is the $i$-th harmonic number (and we use the convention $H_0=0$). If $\gamma$ denotes the Euler-Mascheroni constant, then
\begin{align*}
    |H_i-\ln i-\gamma|\leq\frac{c}{i},
\end{align*}
for some absolute constant $c$, and so we can bound, for all $1\leq j\leq m^*-1$,
\begin{align*}
    \left |\mathbb E[H_{T_j}]-\mathbb E[\ln T_j]-\gamma\right|\leq c\delta,
\end{align*}
where, since $T_j$ stochastically dominates $T_1$, we can take
\begin{align*}
    \delta=\mathbb E\left[\frac{1}{T_1}\right].
\end{align*}
Also, define $\widetilde Z_j$ via
\begin{align*}
T_j=\frac{n^\frac{j}{j+1}}{\beta_{j+1}^{\frac{1}{j+1}}}\widetilde Z^{\frac{1}{j+1}}_j,
\end{align*}
and let $Z$ be an exponential random variable with parameter one. Since $\mathbb E\ln Z=-\gamma$, we have (recall the definition of $\beta_j$ from \eqref{def: betabalance})
\begin{align*}
    \left|\mathbb E[\ln T_j]-\frac{j}{j+1}\ln n+\frac{1}{j+1}\ln \beta_{j+1}+\frac{1}{j+1}\gamma\right|\leq \delta',
\end{align*}
where
\begin{align*}
    \delta'=\sup_{1\leq j\leq m^*-1}\left|\mathbb E\left[\ln \widetilde Z_j\right]-\mathbb E\left[\ln Z\right]\right|
\end{align*}
Therefore, combining these with $\sum_{j=1}^{m^*-1}\frac{j}{j+1}=m^*-H_{m^*}$ and $N=nm$, we obtain
\begin{align*}
   \mathbb E[S^+_{\bold m}]&=m^*(\ln N+\gamma)-(m^*-H_{m^*})(\ln n+\gamma)+\sum_{j=1}^{m^*-1}\frac{1}{j+1}\ln \beta_{j+1}+O(\delta+\delta')\\&=H_{m^*}H_n+m^*\ln m+\sum_{j=1}^{m^*}\frac{1}{j}\ln \beta_j+O(\delta+\delta'+\delta'').
\end{align*}
where $\delta''$ originates from converting $\ln n$ into $H_n$ again. Because of the definitions of $\gamma_j$ and $\beta_j$ from \eqref{def: betabalance}, it remains to bound the three error terms $\delta, \delta', \delta''$ (recall that all implicit constants are allowed to depend on $m^*$ and $\epsilon$). 

The third error $\delta''$ is controlled by $O(\frac{1}{n})$. As for $\delta$, we can use Theorem \ref{thm: univariatebirthday} with the choice of $t^{2}=\frac{n\ln n}{\beta_1}$ (i.e. $\lambda=\ln n$) and obtain
\begin{align*}
    \delta&=\mathbb E\left[\frac{1}{T_1}\right]\\&\leq \frac{1}{t}+\mathbb P(T_1\geq t)\\
    &=O\left(\frac{1}{\sqrt {n\ln n}}+e^{-\ln n}+\frac{\sqrt{\ln n}}{\sqrt n}\right)\\&=O\left(\frac{\sqrt{\ln n}}{n^{\frac{1}{2}}}\right).
\end{align*} 
Finally, we bound $\delta'$. Since $\widetilde Z_j$ is bounded away from $0$ and $\infty$, integration by parts leads to 
\begin{align*}
    \mathbb E[\ln \widetilde Z_j]=\int_{0}^{1}\frac{\mathbb P(\widetilde Z_j>\lambda)-1}{\lambda}d\lambda+\int_1^{\infty}\frac{\mathbb P(\widetilde Z_{j}>\lambda)}{\lambda}d\lambda,
\end{align*}
and $\mathbb E[\ln Z]$ can be written as
\begin{align*}
    \mathbb E[\ln Z]=\int_0^{1}\frac{e^{-\lambda}-1}{\lambda}d\lambda+\int_1^{\infty}\frac{e^{-\lambda}}{\lambda}d\lambda.
\end{align*}
This gives
\begin{align*}
    |\mathbb E[\ln \widetilde Z_j]-\mathbb E[\ln Z]|\leq \int_0^\infty \frac{|\mathbb P(\widetilde Z_j>\lambda)-e^{-\lambda}|}{\lambda}d\lambda.
\end{align*}
Since $\widetilde Z_j=O(n)$, splitting the integral at $M=\ln n$ gives
\begin{align*}
|\mathbb E[\ln \widetilde Z_j]-\mathbb E[\ln Z]|=O\left( \int_0^{M}\frac{|\mathbb P(\widetilde Z_j>\lambda)-e^{-\lambda}|}{\lambda}d\lambda+\mathbb P(\widetilde Z_j>M)\ln n+\frac{e^{-M}}{M}\right).
\end{align*}
By Theorem \ref{thm: univariatebirthday}, we have
\begin{align*}
    \int_0^M\frac{|\mathbb P(\widetilde Z_j>\lambda)-e^{-\lambda}|}{\lambda}d\lambda&=O\left(\int_0^{M}\frac{1}{n^{\frac{1}{j+1}}\lambda^{\frac{j}{j+1}}}+\frac{1}{n^{\frac{j}{j+1}}\lambda^{\frac{1}{j+1}}}d\lambda\right)
    \\&=O\left(\left(\frac{M}{n}\right)^{\frac{1}{m^*}}\right),
\end{align*}
where we use the fact that $e^{-x}$ is Lipschitz for $x>0$ to extend Theorem \ref{thm: univariatebirthday} to all $0\leq \lambda\leq M$ at the cost of the second error term. Theorem \ref{thm: univariatebirthday} also gives
\begin{align*}
    \mathbb P(\widetilde Z_j>M)\ln n=O\left(\ln ne^{-M}+\ln n\left(\frac{M}{n}\right)^{\frac{1}{m^*}}\right).
\end{align*}
Thus,
\begin{align*}
    |\mathbb E[\ln \widetilde Z_j]-\mathbb E[\ln Z]|=O\left[e^{-M}\ln n+
    \ln n\left(\frac{M}{n}\right)^{\frac{1}{m^*}}\right].
\end{align*}
As $M=\ln n$, the result follows.
\end{proof}

\section{Worst strategy: Proof of Theorem \ref{thm: worst case}}
\label{sec: worst strategy}
In this section, we prove Theorem \ref{thm: worst case}. Note that in the analysis of the worst strategy, we define $W_j$ and $T_j$ starting from the top of the deck rather than the bottom. More formally, we are studying $W_j$ and $T_j$ with respect to the deck $(Z_N, \dotsc, Z_1)$ rather than $(Z_1,\dotsc, Z_N)$. Since the distribution is the same, we continue to use the same notation in this section.

\begin{remark}
In the analysis of the worst strategy, what matters are not the random variables $T_j$, which are the last time that no $(j+1)$-tuple is observed, but rather the last times, starting from the bottom of the deck, that there is at least one type of card where no $(j+1)$-tuple is observed. Under the greedy strategy, it is at these times that the card guessed might change. When $\mathbf{m}=m\mathbf{1}$, it so happens that these times are equivalent to the $T_j$ (essentially by flipping the deck upside down), but this is not the case in general. This is why we only consider the case of an even deck.
\end{remark}

We first establish the following exponential tail bounds, strengthening the analogous bounds in \cite{diaconis2020card}.
\begin{proposition}
\label{prop: tail bound}
Consider a deck with multiplicities $\bold m$, and fix $j$ with $1\leq j\leq m^*-1$ and $t$ with $1\leq t\leq N$. Assume that for some $\epsilon>0$, we have $m\geq \epsilon m^*$, and the fraction of cards of types that appear with multiplicity at least $m^*$ is at least $\epsilon$. Let
\begin{equation*}
    \lambda=\frac{t^{j+1}}{n^j}\beta_{j+1}.
\end{equation*}
Then there exists constants $C,C'>0$ depending only on $\varepsilon$ and $m^*$ such that
\begin{equation*}
    \mathbb{P}(T_j\geq t)\leq Ce^{-C'\lambda}.
\end{equation*}
\end{proposition}
\begin{proof}
Lemma \ref{lemma:standardbounds} gives $|\lambda-\mathbb{E}[W_j(t)]|=O\left(\frac{t^{j}}{n^j}\right)$. By taking $C'$ small enough and $C$ large enough, we can assume that $\lambda$ (and thus $t$) is large enough so that $\lambda\geq \frac{1}{2}\mathbb{E}[W_j(t)]$. Taking $x=\frac{1}{2}\mathbb{E}[W_j(t)]$ in Lemma \ref{lem: Tj tail bound} then immediately gives the desired result.
\end{proof}

\begin{proof}[Proof of Theorem \ref{thm: worst case}]
The result is clear when $m=1$ since only the first card has a chance of being guessed correctly, so assume $m>1$.

Let $Y_t$ be the indicator that the $(t+1)$-th guess is correct, and let $J_t$ denote the largest multiplicity of a card type within the top $t$ cards. Notice that under the worst strategy, we have
\begin{equation*}
    \mathbb{P}(Y_t=1|J_t=j)=\frac{m-j}{nm-t},
\end{equation*}
because $J_t=j$ is exactly the event that the type of least multiplicity appears $m-j$ times, among the remaining $nm-t+1$ cards (note that our conventions regarding indexing differ from those in \cite{diaconis2020card} by $1$). Then because $J_t\leq j$ if and only if $T_{j}\geq t$, we have
\begin{equation*}
\begin{split}
    \mathbb{P}(Y_t=1)&=\sum_{j=0}^m\mathbb{P}(Y_t=1|J_t=j)\mathbb{P}(J_t=j)
    \\&=\sum_{j=0}^m\frac{m-j}{nm-t}\mathbb{P}(J_t=j)
    \\&=\sum_{j=0}^{m-1}\frac{1}{nm-t}\mathbb{P}(J_t\leq j)
    \\&=\sum_{j=0}^{m-1}\frac{\mathbb{P}(T_j\geq t)}{nm-t}.
\end{split}
\end{equation*}
Summing over $0\leq t\leq mn-1$, the desired expectation is
\begin{equation*}
    \sum_{j=0}^{m-1}\sum _{t}\frac{\mathbb{P}(T_j\geq t)}{nm-t}.
\end{equation*}

Since $T_0=0$, the $j=0$ terms contribute $O(n^{-1})$ and so may be ignored (since we assumed $m>1$). Now considering the $j$-th term for $j\geq 1$, we sum from $t=1$ to $Cn^{j/(j+1)}\log n$ for some large constant $C$, and cut off the rest of the sum. By Proposition \ref{prop: tail bound}, we have
\begin{equation*}
\begin{split}
    \sum_{t\geq Cn^{j/(j+1)}\log n}\frac{\mathbb{P}(T_j\geq t)}{nm-t}&\leq\sum_{t\geq Cn^{j/(j+1)}\log n}\mathbb{P}(T_j\geq t)
    \\&\leq n\mathbb{P}\left(T_j\geq Cn^{j/(j+1)}\log n\right)
    \\&=O(n^{-1})
\end{split}
\end{equation*}
by taking $C$ large enough.

Using Theorem \ref{thm: Tj approx even univariate}, we have
\begin{equation*}
    \mathbb{P}(T_j\geq t)=e^{-\beta_{j+1}n^{-j}t^{j+1}}+O\left(\frac{t}{n}\right)
\end{equation*}
Thus, as
\begin{equation*}
    \sum _{t\leq Cn^{j/(j+1)}\log n}O\left(\frac{t}{n(nm-t)}\right)=O(n^{-2/(j+1)}\log^2 n),
\end{equation*}
we can replace $\mathbb{P}(T_j\geq t)$ with $e^{-\beta_{j+1}n^{-j}t^{j+1}}$ at the cost of an $O(n^{-2/(j+1)}\log ^2 n)$ error. Finally, we have
\begin{equation*}
\begin{split}
    &\sum_{t\leq Cn^{j/(j+1)}\log n}\frac{\exp({-\beta_{j+1}n^{-j}t^{j+1}})}{nm-t}
    \\=&\sum_{t\leq Cn^{j/(j+1)}\log n}\frac{\exp({-\beta_{j+1}n^{-j}t^{j+1}})}{nm}+O(n^{-2/(j+1)}\log^2 n).
\end{split}
\end{equation*}

Now
\begin{equation*}
    \sum_{t\leq Cn^{j/(j+1)}\log n}\frac{\exp({-\beta_{j+1}n^{-j}t^{j+1}})}{nm}
\end{equation*}
can be approximated with
\begin{equation*}
    \frac{1}{mn}\int_{0}^{Cn^{j/(j+1)}\log n}\exp({-\beta_{j+1}n^{-j}t^{j+1}})dt=\frac{\int_0^{C^{j+1}\beta_{j+1}\log^{j+1}n}\lambda^{\frac{1}{j+1}-1}e^{-\lambda}d\lambda}{(j+1)\beta_{j+1}^{\frac{1}{j+1}}mn^{\frac{1}{j+1}}}
\end{equation*}
at the cost of an $O(n^{-1})$ error. We can bound the tail by
\begin{equation*}
    \int_{C^{j+1}\beta_{j+1}\log^{j+1}n}^\infty\lambda^{\frac{1}{j+1}-1}e^{-\lambda}d\lambda =O\left(n^{-1}\right)
\end{equation*}
by choosing $C$ large enough.

Since
\begin{equation*}
    \frac{1}{j+1}\int_0^{\infty}\lambda^{\frac{1}{j+1}-1}e^{-\lambda}d\lambda=\frac{1}{j+1}\Gamma\left(\frac{1}{j+1}\right)=\Gamma\left(\frac{j+2}{j+1}\right),
\end{equation*}
and the dominant error is $O\left(n^{-\frac{2}{j+1}}\log^2 n\right)$, we have that $j$-th term in the expectation is
\begin{equation*}
    \frac{\Gamma\left(\frac{j+2}{j+1}\right)}{m(\beta_{j+1}n)^{1/(j+1)}}+O\left(n^{-\frac{2}{j+1}}\log^2 n\right).
\end{equation*}
Reindexing the sum, the final estimate is thus
\begin{equation*}
    \sum_{j=2}^{m}\frac{\Gamma\left(\frac{j+1}{j}\right)}{m(\beta_{j}n)^{1/j}}+O\left(n^{-\frac{2}{m}}\log^2 n\right),
\end{equation*}
where the implicit constant can depend on $m$. Only terms for which $j> \frac{m}{2}$ are larger than the error term, so the sum starts from $\lfloor \frac{m}{2}\rfloor +1$. We conclude using \eqref{def: betabalance}.
\end{proof}

\section*{Acknowledgment}
We warmly thank Persi Diaconis for suggesting the problem, Larry Goldstein for pointing out some references, Sam Spiro for suggesting the results on the worst strategy and Xiaoyu He for helpful comments.
 \bibliographystyle{abbrv}
  \bibliography{Bibliography.bib}
\end{document}